\documentclass[preprint,12pt]{elsarticle}
\usepackage{graphicx}
\usepackage[cp1251]{inputenc}
\usepackage{amssymb}
\usepackage{amsthm}
\usepackage{cmap}
\usepackage{picture}

\usepackage{amsmath}

\biboptions{sort&compress}

\newtheorem{theorem}{Theorem}[section]
\newtheorem{lemma}[theorem]{Lemma}

\theoremstyle{example}

\newtheorem{proposition}[theorem]{Proposition}
\newtheorem{corelemma}[theorem]{Core Lemma}
\newtheorem{corollary}[theorem]{Corollary}
\newtheorem{corollaries}[theorem]{Corollaries}

\theoremstyle{definition}
\newtheorem{definition}[theorem]{Definition}

\newtheorem{question}[theorem]{Question}

\journal{...}

\begin{document}

\title{Resolvability in products and squares}

\author{Anton E. Lipin}

\address{Krasovskii Institute of Mathematics and Mechanics, \\ Ural Federal
 University, Yekaterinburg, Russia}

\ead{tony.lipin@yandex.ru}

\begin{abstract} Suppose $X$ and $Y$ are topological spaces, $|X| = \Delta(X)$ and $|Y| = \Delta(Y)$. We investigate resolvability of the product $X \times Y$. We prove that:
	
	I. If $|X| = |Y| = \omega$ and $X,Y$ are Hausdorff, then $X \times Y$ is maximally resolvable;
	
	II. If $2^\kappa = \kappa^+$, $\{|X|, \mathrm{cf}|X|\} \cap \{\kappa, \kappa^+\} \ne \emptyset$ and $\mathrm{cf}|Y| = \kappa^+$, then the space $X \times Y$ is $\kappa^+$-resolvable.
	In particular, under GCH the space $X^2$ is $\mathrm{cf}|X|$-resolvable whenever $\mathrm{cf}|X|$ is an isolated cardinal;
	
	III. ($\frak{r} = \frak{c}$) If $\mathrm{cf}|X| = \omega$ and $\mathrm{cf}|Y| = \mathrm{cf}(\frak{c})$, then the space $X \times Y$ is $\omega$-resolvable.
	If, moreover, $\mathrm{cf}(\frak{c}) = \omega_1$, then the space $X \times Y$ is $\omega_1$-resolvable.
\end{abstract}



\maketitle 


\section{Introduction}

Suppose $\kappa$ is a cardinal. A topological space $X$ is called {\it $\kappa$-resolvable}, if $X$ can be partitioned into $\kappa$-many dense subsets. Clearly, every $\kappa$-resolvable space $X$ satisfies the inequality $\Delta(X) \geq \kappa$, where
$$\Delta(X) = \min\{|U| : U \text{ is a nonempty open subset of } X\}$$
is a {\it dispersion character} of the space $X$. The space $X$ is called {\it maximally resolvable}, if it is $\Delta(X)$-resolvable. If $\kappa \leq \Delta(X)$, but $X$ is not $\kappa$-resolvable, the space $X$ is called {\it $\kappa$-irresolvable}. $2$-resolvable and $2$-irresolvable spaces are also called {\it resolvable} and {\it irresolvable}, respectively \cite{Hewitt1943, Ceder1964}.

A space $X$ is called {\it isodyne}, if $|X| = \Delta(X)$.
Quite often the investigation of resolvability in some class of spaces can be reduced to the case of an isodyne space.
This is explained by the following two observations.
First, if $X$ is a space and $U$ is a nonempty open subset of $X$, then there is a nonempty open subset $V \subseteq U$ such that the subspace $V$ is isodyne.
Second, if every nonempty open subset of a space $X$ contains a nonempty $\kappa$-resolvable subspace, then the space $X$ is $\kappa$-resolvable as well \cite[Theorem 4]{Ceder1964}.

In this paper we focus on a problem of resolvability in products of spaces. The observations from the previous paragraph show that it is sufficient to investigate resolvability in products of isodyne spaces.

It was noted by V.I.~Malykhin in 1973 that the square of an isodyne crowded space is resolvable \cite[text under the proof of Theorem 3]{Malykhin1973}. The following generalization may be technically new.

\begin{theorem}\label{T1}
	If $X$ is an infinite isodyne space, then for every natural $n$ the space $X^n$ is $n!$-resolvable.
\end{theorem}
\begin{proof}
	Choose any enumeration $X = \{x_\alpha : \alpha < |X|\}$.
	For every permutation $f : n \leftrightarrow n$ define $D_f$ as the set of all points $(x_{\alpha_0}, x_{\alpha_1}, \ldots, x_{\alpha_{n-1}}) \in X^n$ such that $\alpha_{f(0)} < \alpha_{f(1)} < \ldots < \alpha_{f(n-1)}$.
	
	It is clear that the sets $D_f$ are pairwise disjoint and their amount is $n!$. Let us show that every $D_f$ is dense in $X^n$.
	Take any nonempty open set of the form $W = U_0 \times U_1 \times \ldots \times U_{n-1}$ in $X^n$. Recursively on $k < n$ choose ordinals $\alpha_{f(k)} < |X|$ such that $x_{\alpha_{f(k)}} \in U_{\alpha_{f(k)}}$ and $\alpha_{f(0)} < \alpha_{f(1)} < \ldots < \alpha_{f(n-1)}$ (this is possible since $|U_{f(k)}| = |X|$).
	The point $(x_{\alpha_0}, x_{\alpha_1}, \ldots, x_{\alpha_{n-1}})$ belongs to $D_f$ and $W$ at the same time.
\end{proof}

Of course, one can easily expand this result to products of distinct isodyne spaces of equal cardinality.

Malykhin constructed under CH a countable crowded $T_1$-space $X$ such that $X^2$ is $4$-irresolvable \cite[Corollary 5]{Malykhin1973}.
In the survey \cite{Pavlov2007} O.~Pavlov asked if there is an infinite Hausdorff (regular) isodyne space $X$ such that $X^2$ is not $3$- or $\omega$-resolvable.
He also asked is there such a countable space.

In 1996 A.~Be{\v s}lagi{\' c} and R.~Levy proved that existence of two (Tychonoff, $0$-dimensional) spaces with irresolvable product is equiconsistent with existence of a measurable cardinal \cite{BL1996}. In 2023 I.~Juh{\'a}sz, L.~Soukup and Z.~Szentmikl{\'o}ssy proved that existstence of $n$ measurable cardinals implies existence of $n+1$ Hausdorff $0$-dimensional spaces whose product is irresolvable \cite{JSS2023}. They also proved the following

\begin{theorem}[Juh{\'a}sz, Soukup, Szentmikl{\'o}ssy, \cite{JSS2023}]\label{TJSS2023}
	The following two statements are consistent modulo a measurable cardinal:
	
	\begin{enumerate}
		\item There is a Hausdorff $0$-dimensional isodyne space $X$ such that $\omega_2 \leq |X| \leq 2^{\omega_1}$ and for any countable space $Y$ the product $X \times Y$ is $\omega_2$-irresolvable.
		
		\item There is a monotonically normal isodyne space $X$ such that $|X| = \aleph_\omega$ and for any countable space $Y$ the product $X \times Y$ is $\omega_1$-irresolvable.
	\end{enumerate}
\end{theorem}

As for infinite products, A.G.~Elkin proved in 1969 that the infinite product of non-single-point spaces is $\frak{c}$-resolvable \cite{Elkin1969}.
Apparently, it is an open problem whether such a product is maximally resolvable.

\section{Resolvability in products of countable spaces}

Let us fix some terminology.
Suppose $X$ and $Y$ are spaces.
We say that sets of the form $M \times \{y\} \subseteq X \times Y$ are {\it horizontal},
whereas sets of the form $\{x\} \times N \subseteq X \times Y$ are {\it vertical}.
If a set $A \subseteq X \times Y$ is horizontal or vertical, we say that $A$ is {\it linear}.
If one of sets $A,B \subseteq X \times Y$ is horizontal and the other is vertical, we say that $A$ and $B$ are {\it orthogonal} and we write $A \perp B$.
If both sets $A$ and $B$ are horizontal or both are vertical, we say that $A$ and $B$ are {\it collinear} and write $A \parallel B$.
Sets of the form $X \times \{y\}$ and $\{x\} \times Y$ are called {\it fibers}.

\begin{theorem}\label{TCountable}
	The product of two countable crowded Hausdorff spaces is maximally resolvable.
\end{theorem}
\begin{proof}
	Suppose $X,Y$ are countable crowded Hausdorff spaces.
	Fix any enumeration $X \times Y = \{p_n : n < \omega\}$.
	
	Recursively on $n < \omega$ we will choose in $X \times Y$ crowded linear sets $K(p_n)$ in such a way that $p_n \in K(p_n)$ and the following two conditions hold for all $m < n < \omega$:
	\begin{enumerate}
		\item[(1)] if $K(p_m) \parallel K(p_n)$, then $K(p_m) \cap K(p_n) = \emptyset$;
		
		\item[(2)] if $K(p_m) \perp K(p_n)$, then the intersection $\overline{K(p_m)} \cap \overline{K(p_n)}$ is either $\{p_n\}$ or empty.
	\end{enumerate}
	
	To see that it is possible, suppose that $n < \omega$ and the sets $K(p_i)$ are chosen for all $i < n$.
	If $n = 0$, let us take as $K(p_0)$ any fiber that contains the point $p_0$. Suppose $n > 0$.
	The property (2) implies that the point $p_n$ cannot belong to the closures of two orthogonal sets of the form $K(p_i)$.
	Hence, we can choose a fiber $L \subseteq X \times Y$ such that $p_n \in L$ and $L \perp K(p_i)$ whenever $p_n \in \overline{K(p_i)}$.
	
	Now, for every $i < n$ let us choose a neighborhood $O_i(p_n)$ of the point $p_n$ in the subspace $L$ in the following way:
	\begin{enumerate}
		\item[I.] If $K(p_i) \subseteq L$, then $O_i(p_n) = L \setminus \overline{K(p_i)}$;
		
		\item[II.] If $K(p_i) \perp L$ and $\overline{K(p_i)} \cap L \ne \{p_n\}$, then $\overline{O_i(p_n)} \cap \overline{K(p_i)} = \emptyset$
		\\ (it is possible since $X$ and $Y$ are Hausdorff).
	\end{enumerate}
	
	Finally, we define $K(p_n) = \bigcap_{i < n} O_i(p_n)$. The sets $K(p_n)$ are constructed.
	
	It follows from the properties (1,2) that for every point $a \in X^2$ there is at most one other point $b$ such that $a \in K(b)$.
	Hence, there is a unique finite sequence $a_0, \ldots, a_n$ such that $a_0 = a$, $a_i \in K(a_{i+1}) \setminus \{a_{i+1}\}$ for $i < n$ and $a_n \notin K(b)$ whenever $b \ne a_n$.
	Let us denote $r(a) = n$.
	
	Since the sets $K(a)$ are crowded, for every open set $W \subseteq X^2$ and any point $a \in W$ there is a point $b \in W$ (which belongs to $K(a) \setminus \{a\}$) such that $r(b) = r(a) + 1$.
	Hence, for every nonempty open set $W \subseteq X^2$ the set $\{r(a) : a \in W\}$ is cofinite in $\omega$.
	
	It remains to take any function $f : \omega \to \omega$ such that for any $n < \omega$ the preimage $f^{-1}(n)$ is infinite and define
	$D_n = \{a \in X^2 : f(r(a)) = n\}$
	for all $n < \omega$. It is clear that the sets $D_n$ are dense and pairwise disjoint.
\end{proof}

\begin{corollary}
	The product of two or more Hausdorff separable crowded spaces is $\omega$-resolvable.
\end{corollary}

\section{Thick subset approach. Main notions}\label{SMain}

Sections \ref{SMain}--\ref{SNeg} are dedicated to infinitary combinatorics.
Corollaries of these results in terms of topology and resolvability are presented in Section \ref{SCorollaries}.

For every set $A$ and cardinal $\kappa$ let us denote $[A]^\kappa = \{M \subseteq A : |M| = \kappa\}$.

Everywhere below, the letters $\kappa, \lambda, \mu, \nu, \tau$ represent cardinals, whereas the letters $\alpha, \beta, \gamma, \delta, \xi$ represent ordinals.

\begin{definition}
	Suppose $A, B$ are sets and $\mu, \nu$ are cardinals.
	We say that a set $E \subseteq A \times B$ is $(\mu,\nu)$-{\it thick}, if for any $M \in [A]^\mu$ and $N \in [B]^\nu$ the intersection $E \cap (M \times N)$ is not empty.
	
	We also say that a set $E \subseteq A \times B$ is {\it thick}, if it is $(|A|, |B|)$-thick.
\end{definition}

We do not exclude the cases $\mu > |A|$ and $\nu > |B|$, although the empty set in these cases turns out to be $(\mu, \nu)$-thick.

Clearly, if $X$ and $Y$ are isodyne spaces, then every thick subset of $X \times Y$ is dense.
Hence, our goal is to partition $X \times Y$ into big number of thick subsets.
This is a natural antipode of Ramsey relation.

One may notice that in the proof of Theorem \ref{T1} we, in fact, show that $A^2$ can be partitioned into $2$ thick subsets for every infinite set $A$.

It is easy to see that the thickness is a monotone property in the following sense.

\begin{proposition}\label{PMono}
	If $E$ is a $(\mu, \nu)$-thick subset of $A \times B$, then for every $E' \supseteq E$, $A' \subseteq A$, $B' \subseteq B$, $\mu' \geq \mu$ and $\nu' \geq \nu$ the set $E' \cap (A' \times B')$ is $(\mu', \nu')$-thick in $A' \times B'$.
\end{proposition}

We use the symbol $\bigsqcup$ to denote the {\it disjoint union}, i.e. the union of pairwise disjoint sets.

\begin{proposition}\label{PCof}
	If $\kappa \times \mathrm{cf}(\lambda)$ or $\mathrm{cf}(\kappa) \times \lambda$ can be partitioned into $\tau$-many thick subsets, then so is $\kappa \times \lambda$.
\end{proposition}
\begin{proof}
	Without loss of generality we may assume the first case.	
	Choose any partition $\lambda = \bigsqcup_{\beta < \mathrm{cf}(\lambda)} B_\beta$ such that $|B_\beta| < \lambda$.
	Take any partition $\kappa \times \mathrm{cf}(\lambda) = \bigsqcup_{\gamma < \tau} H_\gamma$ into thick subsets.
	For all $\gamma < \tau$ define $E_\gamma = \bigsqcup_{(\alpha,\beta) \in H_\gamma} \{\alpha\} \times B_\beta$.
	It is obvious that the sets $E_\gamma$ are as required.
\end{proof}

If $f : X \to Y$ and $A \subseteq X$, then we denote $f[A] = \{f(x) : x \in A\}$.

\begin{definition}
	Let us say that a family $\mathcal{E}$ of subsets of a set $X$ is $\tau$-{\it breakable}, if there is a function $f : X \to \tau$ such that $f[A] = \tau$ whenever $A \in \mathcal{E}$.
\end{definition}

We are ready to formulate and prove

\begin{corelemma}\label{LCore}
	Suppose $\mathrm{cf}\left( \kappa^\mu \right) = \lambda$ and every subfamily of $[\kappa]^\mu$ of cardinality less than $\kappa^\mu$ is $\tau$-breakable. Then $\kappa \times \lambda$ can be partitioned into $\tau$-many $(\mu,\lambda)$-thick subsets.
	
	If, moreover, $\tau^+ = \lambda$, then $\kappa \times \lambda$ can be partitioned into $\lambda$-many $(\mu,\lambda)$-thick subsets.
\end{corelemma}
\begin{proof}
	Let us enumerate $[\kappa]^\mu = \{A_\gamma : \gamma < \kappa^\mu\}$.
	Choose any increasing sequence $(\delta_\beta)_{\beta < \lambda}$, whose limit is $\kappa^\mu$.
	
	\medskip
	
	\underline{Partition into $\tau$-many subsets.}
	
	For every $\beta < \lambda$ take a function $f_\beta : \kappa \to \tau$ such that $f_\beta[A_\gamma] = \tau$ whenever $\gamma < \delta_\beta$.
	For all $\xi < \tau$ define $E_\xi = \{(\alpha, \beta) \in \kappa \times \lambda : f_\beta(\alpha) = \xi\}$.
	Let us show that the sets $E_\xi$ are as required.
	First, they are pairwise disjoint.
	Now suppose $M = A_\gamma \in [\kappa]^\mu$ and $N \in [\lambda]^\lambda$.
	Take $\beta \in N$ such that $\delta_\beta > \gamma$.
	We have $f_\beta[A_\gamma] = \tau$, so the set $A_\gamma \times \{\beta\} \subseteq M \times N$ has nonempty intersection with every $E_\xi$.
	
	\medskip
	
	\underline{Partition into $\lambda$-many subsets for $\tau^+ = \lambda$.}
	
	For every $\beta < \lambda$ take a function $f_\beta : \kappa \to \lambda$ such that $f_\beta[A_\gamma] \supseteq \beta$ whenever $\gamma < \delta_\beta$.
	For all $\xi < \lambda$ define $E_\xi = \{(\alpha, \beta) \in \kappa \times \lambda : f_\beta(\alpha) = \xi\}$.
	The sets $E_\xi$ are pairwise disjoint.
	Let us show that every $E_\xi$ is $(\mu, \lambda)$-thick.
	Suppose $M = A_\gamma \in [\kappa]^\mu$ and $N \in [\lambda]^\lambda$.
	Take $\beta \in N$ such that $\beta > \xi$ and $\delta_\beta > \gamma$.
	We have $\xi \in f_\beta[A_\gamma]$, so the set $A_\gamma \times \{\beta\} \subseteq M \times N$ has nonempty intersection with $E_\xi$.
\end{proof}

Now we only have to find some cardinals satisfying the conditions of Core Lemma \ref{LCore}.

\section{Thick subset approach. Positive results}

\subsection{Corollaries of Kuratowski Lemma}

The following statement sometimes is referred as Kuratowski Lemma.

\begin{lemma}\label{LKur}
Suppose $\mathcal{E}$ is a family of sets, $\kappa$ is an infinite cardinal, $|\mathcal{E}| \leq \kappa$ and $|A| \geq \kappa$ whenever $A \in \mathcal{E}$.
Then $\mathcal{E}$ is $\kappa$-breakable.
\end{lemma}
\begin{proof}
	Let us index $\mathcal{E} = \{A_\alpha : \alpha < \kappa\}$.
	Recursively on $\beta$, choose pairwise distinct points $x_\alpha^\beta \in A_\alpha$ and $x_\beta^\alpha \in A_\beta$ for all $\alpha \leq \beta < \kappa$.
	Finally, denote $f(x_\alpha^\beta) = \beta$ for all $\alpha, \beta < \kappa$.
	It does not matter how to define $f$ in all other points.
	We have $f[A] = \kappa$ for all $A \in \mathcal{E}$.
\end{proof}

Kuratowski Lemma \ref{LKur} and Core Lemma \ref{LCore} imply

\begin{proposition}\label{PThick1}
	If $2^\kappa = \kappa^+ = \lambda$, then $\lambda \times \lambda$ can be partitioned into $\lambda$-many $(\kappa, \lambda)$-thick subsets.
\end{proposition}
\begin{proof}
	Note that $\lambda^\kappa = \lambda$.
	Applying Core Lemma \ref{LCore} to $\overline{\kappa} = \overline{\lambda} = \lambda$ and $\overline{\mu} = \overline{\tau} = \kappa$,
	we obtain the required statement.
\end{proof}

Applying Proposition \ref{PMono}, we obtain two corollaries of Proposition \ref{PThick1}.

\begin{corollaries}\label{C1} If $2^\kappa = \kappa^+ = \lambda$, then:
	\begin{enumerate}
		\item[(1)] $\lambda \times \lambda$ can be partitioned into $\lambda$-many thick subsets;
		
		\item[(2)] $\kappa \times \lambda$ can be partitioned into $\lambda$-many thick subsets.
	\end{enumerate}
\end{corollaries}

Proposition \ref{PCof} allows us to deduce from Corollaries \ref{C1} the following

\begin{theorem}\label{TKur}
	If $2^\kappa = \kappa^+$, $\{|A|, \mathrm{cf}|A|\} \cap \{\kappa, \kappa^+\} \ne \emptyset$ and $\mathrm{cf}|B| = \kappa^+$, then the product $A \times B$ can be partitioned into $\kappa^+$-many thick subsets.
\end{theorem}

\subsection{Corollaries of MA. Unsplitting number}

Let us recall the following notions from \cite[Definitions 3.1 and 3.6]{Blass2010}.

\begin{definition}
	A set $S \subseteq \omega$ {\it splits} a set $A \subseteq \omega$, if both $A \cap S$ and $A \setminus S$ are infinite.
	A family $\mathcal{R} \subseteq [\omega]^\omega$ is {\it unsplittable}, if no single set splits all members of $\mathcal{R}$.
	The {\it unsplitting number} $\frak{r}$ is the smallest cardinality of any unsplittable family.
\end{definition}

Martin's axiom implies $\frak{r} = \frak{c}$ \cite{Blass2010}.

\begin{proposition}\label{Plessr}
	If $\mathcal{H} \subseteq [\omega]^\omega$ and $|\mathcal{H}| < \frak{r}$, then $\mathcal{H}$ is $\omega$-breakable.
\end{proposition}
\begin{proof}
	Denote $M_0 = \omega$ and suppose that for some $n < \omega$ we have a set $M_n \subseteq \omega$ such that $|M_n \cap A| = \omega$ whenever $A \in \mathcal{H}$.
	Denote $\mathcal{H}_n = \{A \cap M_n : A \in \mathcal{H}\}$.
	Since $|\mathcal{H}_n| < \frak{r}$, there is a set $S_n \subseteq M_n$ which splits the family $\mathcal{H}_n$.
	Denote $M_{n+1} = S_n$ and $R_n = M_n \setminus S_n$.
	
	Clearly, we just constructed for all $n < \omega$ pairwise disjoint sets $R_n \subseteq \omega$ such that $A \cap R_n \ne \emptyset$ whenever $A \in \mathcal{H}$.
	It remains to define $f(x) = n$ for all $x \in R_n$.
	It does not matter how to define the function $f$ on the set $\omega \setminus \bigsqcup_{n < \omega} R_n$.
\end{proof}

Proposition \ref{Plessr} and Core Lemma \ref{LCore} imply the following

\begin{proposition}
	The product $\omega \times \mathrm{cf}(\frak{c})$ can be partitioned into:
	\begin{enumerate}
		\item[] \rm{($\frak{r} = \frak{c}$)} $\omega$-many thick subsets;
		
		\item[] \rm{($\frak{r} = \frak{c}$ and $\mathrm{cf}(\frak{c}) = \omega_1$)} $\omega_1$-many thick subsets.
	\end{enumerate}
\end{proposition}
\begin{proof}
	Here we apply Core Lemma \ref{LCore} to $\overline{\kappa} = \overline{\mu} = \overline{\tau} = \omega$ and $\overline{\lambda} = \mathrm{cf}(\frak{c})$.
\end{proof}

Using Proposition \ref{PCof}, we obtain the following result.

\begin{theorem}[$\frak{r} = \frak{c}$]\label{TUnspl}
	If $\mathrm{cf}(\kappa) = \omega$ and $\mathrm{cf}(\lambda) = \mathrm{cf}(\frak{c})$, then the product $\kappa \times \lambda$ can be partitioned into $\omega$-many thick subsets.
	
	If, moreover, $\mathrm{cf}(\frak{c}) = \omega_1$, then the product $\kappa \times \lambda$ can be partitioned into $\omega_1$-many thick subsets.
\end{theorem}

\section{Thick subset approach. Negative results}\label{SNeg}

Results of this section give no information on resolvability, they only demonstrate restrictions of our method.

Let us recall \cite[p.~107]{Jech2006}

\begin{theorem}[Ramsey]\label{TRamsey}
	If $n < \omega$, then for every function $f : [\omega]^2 \rightarrow n$ there is a set $A \in [\omega]^\omega$ and a number $k < n$ such that $f(E) = k$ whenever $E \in [A]^2$.
\end{theorem}

\begin{corollary}\label{POmega}
	The product $\omega \times \omega$ cannot be partitioned into $3$ thick subsets.
\end{corollary}
\begin{proof}
	Suppose $h : \omega \times \omega \to 3$.
	We will show that at least one of the sets $h^{-1}(k)$ for $k < 3$ is not thick.
	
	For every $m < n < \omega$ we define $f\{m,n\} = h(m, n)$ and $g\{m, n\} = h(n, m)$.
	By Ramsey's theorem \ref{TRamsey} there is a set $A \in [\omega]^\omega$ and a number $i < 3$ such that $f(E) = i$ for all $E \in [A]^2$.
	Again by Ramsey's theorem there is a set $B \in [A]^\omega$ and a number $j < 3$ such that $g(E) = j$ for all $E \in [B]^2$.
	Finally, let us partition $B$ into two infinite subsets $K$ and $L$.
	It is easy to see that for all $m \in K$ and $n \in L$ we have either $h(m, n) = i$ or $h(m, n) = j$.
\end{proof}

The same is true for a weakly compact cardinal instead of $\omega$ \cite[Definition 9.8; the arrow notation is defined on page 109]{Jech2006}.

One cannot replace $\omega$ by a singular cardinal in the Ramsey's theorem. However,

\begin{theorem}
	Suppose $\mathrm{cf}(\kappa) = \omega$ and $2^\mu < \kappa$ whenever $\mu < \kappa$.
	Then the product $\kappa \times \kappa$ cannot be partitioned into $3$ thick subsets.
\end{theorem}
\begin{proof}
	By Corollary \ref{POmega} we may assume that $\kappa > \omega$.
	Fix an arbitrary partition $\kappa = \bigsqcup_{n < \omega} A_n$ such that the cardinals $\mu_n = |A_n|$ are regular, $\mu_n > \frak{c}$ and $2^{\mu_n} < \mu_{n+1} < \kappa$ for all $n < \omega$.
	Suppose $h : \kappa \times \kappa \to 3$.
	As before, we will show that at least one of the sets $h^{-1}(i)$ for $i < 3$ is not thick.
	
	For every $n < \omega$ and all $x \in A_n$ denote by $f_x$ the function $\bigsqcup_{k < n}A_k \to 3$ defined as $f_x(y) = h(x,y)$.
	Clearly, $|\{f_x : x \in A_n\}| \leq 2^{\mu_0 + \ldots + \mu_{n-1}} < \mu_n$,
	so there is a function $\overline{f}_n$ such that the set $K_n = \{x \in A_n : f_x = \overline{f}_n\}$ has cardinality $\mu_n$.
	In a similar way, for all $y \in A_n$ denote by $g_y$ the function $\bigsqcup_{k < n}A_k \to 3$ defined as $g_y(x) = h(x,y)$
	and note that there is a function $\overline{g}_n$ such that the set $L_n = \{y \in A_n : g_y = \overline{g}_n\}$ has cardinality $\mu_n$.
	
	Denote $\mathbb{N}_{>n} = \omega \setminus (n+1) = \{n+1, n+2, \ldots\}$.	
	For every $n < \omega$ and all $x \in K_n$ denote by $\varphi_x$ the function $\mathbb{N}_{>n} \to 3$ defined as $\varphi_x(k) = h(x,y)$, where $y \in L_k$.
	By the choice of $L_k$, this does not depend on $y$.
	Since $\mu_n > \frak{c}$, there is a function $\overline{\varphi}_n$ such that the set $P_n = \{x \in K_n : \varphi_x = \overline{\varphi}_n\}$ has cardinality $\mu_n$.
	In a similar way, for all $y \in L_n$ denote by $\psi_y$ the function $\mathbb{N}_{>n} \to 3$ defined as $\psi_y(k) = h(x,y)$, where $x \in L_k$,
	and note that there is a function $\overline{\psi}_n$ such that the set $Q_n = \{y \in L_n : \psi_y = \overline{\psi}_n\}$ has cardinality $\mu_n$.
	
	By the choice of the sets $P_m$ and $Q_n$, the function $h$ is constant on the set $P_m \times Q_n$ whenever $m \ne n$.
	Define the function $\chi : \omega \times \omega \to 3$ in the following way: $\chi(m,n) = h(x,y)$, where $x \in P_{2m}$ and $y \in Q_{2n+1}$.
	By Corollary \ref{POmega}, there are infinite sets $M,N \subseteq \omega$ and a number $i < 3$ such that $(M \times N) \cap \chi^{-1}(i) = \emptyset$.
	Hence, for the sets $U = \bigsqcup_{m \in M} P_{2m}$ and $V = \bigsqcup_{n \in N} Q_{2n+1}$ we have $(U \times V) \cap h^{-1}(i) = \emptyset$, whereas $|U| = |V| = \kappa$.	
\end{proof}

\begin{question}
	Is it consistent with ZFC that $\aleph_\omega \times \aleph_\omega$ can be partitioned into $3$ thick subsets? Into $\aleph_\omega$-many thick subsets?
\end{question}

\begin{question}
	Is it consistent with ZFC that $\omega_1 \times \omega_1$ cannot be partitioned into $3$ thick subsets? Into $\omega_1$-many thick subsets?
\end{question}

\begin{question}
	Are there infinite cardinals $\kappa > \mu$ and $\lambda > \nu$ such that $\kappa \times \lambda$ can be partitioned into $2$ $(\mu, \nu)$-thick subsets?
\end{question}

\section{Corollaries in terms of resolvability}\label{SCorollaries}

The following result is an immediate corollary of Theorem \ref{TKur}.

\begin{theorem}\label{T3}
	If $2^\kappa = \kappa^+$, $X$ and $Y$ are isodyne spaces, $\{|X|, \mathrm{cf}|X|\} \cap \{\kappa, \kappa^+\} \ne \emptyset$ and $\mathrm{cf}|Y| = \kappa^+$, then the space $X \times Y$ is $\kappa^+$-resolvable.
\end{theorem}

\begin{corollaries}[GCH]
Suppose $X$ and $Y$ are isodyne spaces. Then:
\begin{enumerate}
	\item $X \times Y$ is maximally resolvable, providing that $|X| = |Y|$ is an isolated cardinal;
	
	\item $X \times Y$ is $\mathrm{cf}|X|$-resolvable, providing that $\mathrm{cf}|X| = \mathrm{cf}|Y|$ is an isolated cardinal;
	
	\item $X \times Y$ is maximally resolvable, providing that the cardinals $|X|$ and $|Y|$ are neighboring.
	
	\item $X \times Y$ is $\omega_1$-resolvable, providing that $|X| = \aleph_\omega$ and $|Y| = \omega_1$.
\end{enumerate}
\end{corollaries}

The third corollary was also obtained by V.I.~Malykhin in \cite[Corollary 2]{Malykhin1973}.
The fourth corollary may be of interest, considering the second point of Theorem \ref{TJSS2023} of I.~Juh{\'a}sz, L.~Soukup and Z.~Szentmikl{\'o}ssy.

Theorem \ref{TUnspl} implies the following

\begin{theorem}[$\frak{r} = \frak{c}$]
	If $X$ and $Y$ are isodyne spaces, $\mathrm{cf}|X| = \omega$ and $\mathrm{cf}|Y| = \mathrm{cf}(\frak{c})$, then the space $X \times Y$ is $\omega$-resolvable.
	
	If, moreover, $\mathrm{cf}(\frak{c}) = \omega_1$, then the space $X \times Y$ is $\omega_1$-resolvable.
\end{theorem}

\begin{question}
	Suppose $X$ is an isodyne (Hausdorff, regular, Tychonoff...) space of cardinality $\aleph_\omega$.
	Is $X^2$ $3$-resolvable? Is it maximally resolvable?
\end{question}

\begin{question}
	Is it consistent with ZFC that for some isodyne (Hausdorff, regular, Tychonoff...) space $X$ of cardinality $\omega_1$ the square $X^2$ is not $3$- (or maximally) resolvable?
\end{question}



\bibliographystyle{model1a-num-names}
\bibliography{<your-bib-database>}

\end{document}